\documentclass[12pt]{article}

\usepackage{amsmath}
\usepackage{amsfonts}
\usepackage{enumerate}
\usepackage{graphicx}
\numberwithin{equation}{section}

\newtheorem{lemma}{Lemma}
\newtheorem{theorem}{Theorem}

\newtheorem{proof}{Proof}

\begin{document}

\author{Rafik Aramyan}

\title{An approach to the spherical mean
Radon transform with detectors on a line}

\maketitle

affiliation: {1.  Russian-Armenian University

2. Institute of Mathematics NAS RA.
}

e-mail: {rafikaramyan@yahoo.com}

Keywords: Tomography, thermoacoustic tomography, spherical Radon transform, inverse problem.

\begin{abstract}  The article suggests a new approach what is called a
consistency method for the inversion of the spherical Radon transform in 2D with detectors on a line. It is known that there is not an exact inversion formula in 2D. By means of the method was proved that the reconstruction has a local description and found
a new iteration formula  which give an practical algorithm to
recover an unknown function supported completely on one side of a line $L$  from its spherical means over circles
centered on the line $L$. Such an inversion is required in problems of thermo- and photo-acoustic tomography.
\end{abstract}

\section{Introduction and formulation of the problem}\label{S1}

\noindent Medical tomography has had a huge impact on medical diagnostics.
The classical Radon transform maps a function to its integrals over straight lines and  serves as the basis of x-ray Computer Tomography.
Recently researchers have been developing novel methods that combine different
physical types of signals. The most successful example of such
a combination is the thermoacoustic tomography (TAT).  Thermoacoustic theory has been discussed in many literature
reviews such as \cite{FiRa}-\cite{XW}. Briefly TAT procedure is: a short-duration electromagnetic (EM) pulse is sent through a biological object
with the aim of triggering a thermoacoustic
response in the tissue. The amount of energy absorbed at a location $X$ strongly
depends on the local biological properties of the cells.
Thus, if the energy absorption distribution function $f$ were known, it would provide
a great diagnostic tool.
The acoustic wave which is the result of the thermoelastic expansion can be measured by transducers placed outside the object (assuming the sound speed c constant). Thus, one effectively measures the integrals of $f$
over all spheres centered at the transducers’ locations. To recover $f$
one needs to invert the so-called spherical Radon transform
 of $f$  that integrates a function over all such spheres.

We denote by $\mathbf R^n$  ($n\geq2$) the Euclidean $n$ - dimensional space. Let ${\mathbf S^{n-1}}$ be the $n-1$ dimensional unit sphere in  $\mathbf R^n$ with the center at the origin  $O\in\mathbf R^n$, $\sigma_{n-1}$ its total surface measure. By $S(p,r)$ we denote the sphere of radius $r > 0$  centered at $p\in\mathbf R^n$.

\noindent The above motivated the study of the following mathematical problem. For a continuous, real valued function $f$ supported in a compact region $G$, we are interested in recovering $f$ from the mean value $Mf(p,r)$ of $f$ over spheres $S(p,r)$ centered on
$L$; that is, given $Mf(p,r)$ for all $p\in L$ and  $r>0$, we wish to recover $f$.

\noindent In order to implement the TAT reconstruction  the following problems arise. For which sets $ L$ the data collected by transducers placed along
$L$ is sufficient for unique reconstruction of $f$ (set $L$ is called a set of injectivity if the transform \eqref{2} is injective) and what are inversion formulas.

\noindent  Agranovsky and Quinto in \cite{AgQu}, \cite{AgKu}  have proved several significant uniqueness results
for the spherical Radon transform. In \cite{AgQu} they gave a complete characterization
of sets of uniqueness (sets of centers) for the circular Radon transform on
compactly supported functions in the plane. In \cite{AmKu1}, was provided a complete range description in case of circular Radon transform in 2D.

\noindent Obviously any line L (or a hyperplane in higher dimensions) is a non-uniqueness set,
since any function $f$ odd with respect to $L$ will clearly produce no signal: $Mf(p,r)=0$.
On the other hand (see \cite{CH}, \cite{F}), if
$f$ is supported completely on one side of the the line $L$ (the standard situation in TAT), it
is uniquely recoverable from its spherical means centered on $L$, and thus from the observed data.

Exact inversion formulas for the spherical Radon transform are currently known
for boundaries of special domains, including spheres, cylinders and hyperplanes (\cite{An}, \cite{AmKu}, \cite{De}, \cite{No},  \cite{No1}, \cite{FiHa},  \cite{FiPa}, \cite{CH}).

\noindent In this paper for a continuous, real valued function $f$ defined in $\mathbf R^2$ and supported completely on one side of a line $L$, we are interested in recovering $f$ from the mean value of $f$ over circles centered on $L$.
The article suggests a new approach what is called a
consistency method for the inversion of the spherical Radon transform in 2D with detectors on a line. By means of the method
a new iteration formula was found which give an practical algorithm to
recover an unknown function supported in a compact region from its spherical means over circles
centered on a line outstand the region. Also was proved that reconstruction has a local description (see Theorem 1 below).

\noindent The consistency method, suggested by the author of the paper, first was applied in \cite{Ara10} (see also \cite{Ara09}, \cite{Ara11})
to inverse generalize Radon transform on the sphere.

\noindent Note, one can apply the consistency method to inverse the spherical Radon transform for dimensions $n\geq3$. Also, one can consider to apply the consistency method to inverse the spherical Radon transform for different geometries of transducers.

\noindent Now we consider the circular Radon transform on the plane. For a continuous function $f$ supported in the compact region $G\in\mathbf R^2$  we have (see \eqref{1})
\begin{equation}\label{2}
Mf(p,r)=\frac{1}{2\pi}\int_{\mathbf S^{1}}f(p+r\omega)\,d\varphi,\,\,\,\,\emph{for}\,\,\,\,(p,r)\in L\times[0,\infty).
\end{equation}
Here $d\varphi$ is the circular Lebesgue measure on ${\mathbf S^{1}}$, $\omega=(\sin\varphi, \cos\varphi)$. The value
$Mf(p,r)$ is the average of $f$ over the circle $S(p,r)$ with center $p\in L$ and radius $r > 0$.

\noindent  We consider  the restriction of  $f$  onto the circle
${S}(p,r)$ for $(p,r)\in L\times[0,\infty)$.

\noindent A pair say $(S(p,r),q)$ where $q\in{S}(p,r)$ we call a circular flag (in integral geometry there is a concept of a flag an ordered pair of orthogonal
unit vectors \cite{b1},\cite{b2}). There are two equivalent representations of a circular flag $(S(p,r),q)$  where $q\in{S}(p,r)$, dual each other:
\begin{equation}\label{3}
(p,r,\varphi)\,\quad\text{and}\quad\,(x,y,\phi),
\end{equation}
where $p\in L$ and $\varphi\in[-\pi,\pi]$ is the angular coordinate of $q\in{S}(p,r)$ measured from the direction perpendicular to $L$, while $x,y$ are the  Euclidean coordinates of $q$ and
$\phi$ is the direction (the angular coordinate) of $\overrightarrow{qp}$.

\noindent Thus one can represent the restriction of $f$ onto the circle
${S}(p,r)$ by $f(p,r,\varphi)$, where $\varphi\in[-\pi, \pi]$ is the angular coordinate of $q\in {S}(p,r)$.

\noindent  The idea of the method is the following: for every $(p,r)$  the equation \eqref{2} reduces to an
integral equation on the circle $S(p,r)$. The general solution of the reduced integral equation we write in terms of Fourier series expansion with unknown coefficients. Let $G(p,r,\cdot)$ be a solution of the reduced integral equation for $(p,r)$.

\noindent{\emph{ \textbf{Definition }1.}} If $G$ written in
dual coordinates satisfies
\begin{equation}\label{4}
G(x,y,\phi)\,=\,G(x,y)\end{equation} (no dependence on the
variable $\phi$), then $G$ is called {\it a consistent solutions}.

\noindent There is a principle: {\it each consistent solutions
$G$ produces via the map
\begin{equation}\label{5}
G(p,r,\varphi)\to
G(x,y,\phi)=G(x,y)=f(x,y)\end{equation}
the solution of \eqref{2}, and vice versa. Conversely, restrictions of the
solution of \eqref{2} onto the circles $S(p,r)\, $, ($(p,r)\in L\times[0,\infty)$) is a consistent
solutions.}

\noindent Hence the problem of finding the solution of \eqref{2} reduces
to finding the consistent solutions of the reduced equations \eqref{2}.

\noindent In this paper was proved the following theorem. Let $f$ be a
continuous, real valued function supported in the compact region $G$ located on one side of the the line $L$. On the plane consider usual cartesian system of coordinate choosing $L$ as  the $x$-axis. $Mf(p,r)$ is the average of $f$ over a circle with center $p\in L$ and radius $r > 0$.

\begin{theorem} Let $f$ be a
continuous, real valued function supported in the compact region $G$ located on one side of the line $L$. For $(x,y)\in G$ the value $f(x,y)$ depends on values Mf on a neighborhood of $p=(x,0)\in L$ and $0\leq r\leq y$.
\end{theorem}

\noindent Now we describe the inversion formula. We define a sequence of
{\it standard} polynomials $Z_{n,i}$ defined on the interval $[0,1]$, where $n,i$ are integers and $0\leq i\leq n$:
\begin{equation}\label{8}
Z_{n,i}(t)=\sum_{j=1}^{n+i}z_j(2n,2i)\,t^{2j-1},\,\,\,\,\,\,\,t\in[0,1]\end{equation}
with coefficients
\begin{equation}\label{9}
z_j(2n,2i)=\sum_{k=i,\,\, j-i\leq k}^{n}A_j(2k,2i)\,\,\,\,\,\emph{for}\,\,\,\,\,1\leq j\leq n+i. \end{equation}
In \S 5 was found recurrent relations by means of which one can find the coefficients $A_j(2k,2i)$ for integers $k\geq1$,  $0\leq i\leq k$ and $1\leq j\leq k+i$.
We call $Z_{n,i}$ standard polynomials because their construction does not depend on $f$.

\begin{theorem} Let $f$ be an infinitely differentiable
 real valued function  $f$ supported in the compact region $G$ located on one side of the line $L$. For $(x,y)\in G$ we have
\begin{equation}\label{10}
f(x,y)=\lim_{n\to \infty}\left(2(n+1)Mf((x,0),y)+\sum_{i=0}^{n}\int_{0}^{y}y^{2i-1}Z_{n,i}(\frac{u}{y})\,(Mf((x,0),u))^{(2i)}_x\,du\right),\end{equation}
here $(Mf((x,0),u))^{(2i)}_x$  is the derivative of order $2i$ with respect the variable $x$

\noindent ($(Mf((x,0),u))^{(0)}_x=Mf((x,0),u)$).
\end{theorem}

\noindent Theorem 2 suggests a practical algorithm to reconstruct $f$.

\section{General solution of the reduced equation
\eqref{2}}\label{S2}

\noindent For a fix $(p,r)\in L\times[0,\infty)$ the restriction of $f$ onto the circle $S(p,r)$ we write in the form
\begin{equation}\label{10}f(p,r,\varphi),\,\,\,\,\,\,\, \varphi\in[-\pi, \pi]\end{equation}
where $\varphi$ is the angular coordinate of $q\in S(p,r)$ measured from the direction perpendicular to $L$. On the plane we consider usual cartesian system of coordinate choosing $L$ as  the $x$-axis and below the point $(p,0)\in L$ we will identify with $p$.

\noindent It is known that  periodic, continuous, with piecewise-continuous first-derivative function $f$  can be written as its Fourier series expansion.
For any $(p,r)$ the Fourier series expansion of the restriction $f(p,r,\varphi)$ is
\begin{equation}\label{11}
f(p,r,\varphi)=\sum_{k=0}^{\infty} \left(a_{k}(p,r) \cos{k\,\varphi}  + b_{k}(p,r) \sin{k\,\varphi} \right) .\end{equation}

\noindent Taking into account \eqref{2} we have
\begin{equation}\label{12}
f(p,r,\varphi)=Mf(p,r)+\sum_{k=1}^{\infty} \left(a_{k}(p,r) \cos{k\,\varphi}  + b_{k}(p,r) \sin{k\,\varphi} \right).\end{equation}

\noindent Now we are going to write $f(p,r,\varphi)$ in dual coordinates.

\noindent The transform $(p,r,\varphi) \longrightarrow (x,y,\phi)$ (see \eqref{3}) can be represented by the following system
\begin{equation}\label{13}\begin{cases}
x=p+r \cos\phi\\
y= r \sin\phi\\
\phi=\varphi +\frac{\pi}{2} \end{cases}
\end{equation}
We denote the (partial) derivative of a function $f$ with
respect to a variable say $v$ by $f'_v$.
From \eqref{13} we get the following expressions for the derivatives
\begin{equation}\label{14}
\varphi'_\phi=1,\quad
r'_\phi=\frac{r \sin\varphi}{\cos\varphi},\quad
p'_\phi=\frac{r}{\cos\varphi}.
\end{equation}

\section{The consistency condition}\label{S3}

\noindent Now we consider the coefficients $a_{k}(p,r)$, \,$b_{k}(p,r)$ ($k=1,2,...$) in
\eqref{12} as functions of $(p,r)\in(-\infty,\infty))\times[0,\infty)$ and try to find them from the consistency condition. We
write $f(p,r,\varphi)$ in dual coordinates and require that the right side
should not depend on $\phi$ for every $(x,y)\in\mathbf R^2$. It follows from \eqref{13} the following theorem.

\begin{theorem} Let $f$ be an infinitely differentiable
 real valued function  $f$ supported in the compact region $G$ located on one side of the line $L$. We have
\begin{equation}\label{14.1}
f_r r\,\sin\varphi + f_p r +f_\varphi (\cos\varphi-r)=0.\end{equation}
\end{theorem}

\noindent Proof of the Theorem 3 follows  from \eqref{14} and the condition that \begin{equation}\label{14.2}
(f(x,y))'_\phi=(f(p,r,\varphi))'_\phi=0.\end{equation}

\noindent
All Fourier coefficients of the function on the left side of \eqref{14.1} equals $0$ as the function identity equals $0$. We have
\begin{equation}\label{15}
\int_{-\pi}^{\pi}\left(f_r r\,\sin\varphi + f_p r +f_\varphi (\cos\varphi-r)\right)\cos{k\,\varphi}\, d\varphi =0.\end{equation}
and
\begin{equation}\label{16}
\int_{-\pi}^{\pi}\left(f_r r\,\sin\varphi + f_p r +f_\varphi (\cos\varphi-r)\right)\sin{k\,\varphi}\, d\varphi =0.\end{equation}
Substituting  \eqref{12} into \eqref{15} we get
\begin{multline}\label{17}
\int_{-\pi}^{\pi}\left([Mf(p,r)+\sum_{i=1}^{\infty} (a_{i}(p,r) \cos{i\,\varphi}  + b_{i}(p,r) \sin{i\,\varphi})]^\prime_r \,r\sin\varphi +\right.\\ r\,[Mf(p,r)+\sum_{i=1}^{\infty} (a_{i}(p,r) \cos{i\,\varphi}  + b_{i}(p,r) \sin{i\,\varphi})]^\prime_p +\\
\left.[Mf(p,r)+\sum_{i=1}^{\infty} (a_{i}(p,r) \cos{i\,\varphi}  + b_{i}(p,r) \sin{i\,\varphi})]^{\prime}_\varphi (\cos\varphi-r)\right)\,\cos{k\,\varphi}\,d\varphi =0.\end{multline}
From \eqref{17} we obtain the following differential equation for $b_{k}(p,r)$.

\noindent For $k=1$
\begin{equation}\label{18}r (b_1(p,r))^\prime_r + b_1(p,r)=-2r\,(Mf(p,r))^\prime_p;\end{equation}
for $k=2$
\begin{equation}\label{19}r (b_2(p,r))^\prime_r + 2\,b_2(p,r)+2\,r (a_{1}(p,r))^\prime_p =0;\end{equation} for $k>2$
\begin{equation}\label{20}r (b_k(p,r))^\prime_r + k\,b_k(p,r)-r (b_{k-2}(p,r))^\prime_r+(k-2)b_{k-2}(p,r)+2\,r (a_{k-1}(p,r))^\prime_p =0.\end{equation}

\noindent By analogous way substituting \eqref{12} into \eqref{16} for $a_{k}(p,r)$ we get.  For $k=2$
\begin{equation}\label{19}r (a_2(p,r))^\prime_r + 2\,a_2(p,r)-2\,r (b_{1}(p,r))^\prime_p =2r\,(Mf(p,r))^\prime_r;\end{equation} for $k>2$
\begin{equation}\label{20}r (a_k(p,r))^\prime_r + k\,a_k(p,r)-r (a_{k-2}(p,r))^\prime_r+(k-2)a_{k-2}(p,r)-2\,r (b_{k-2}(p,r))^\prime_p =0.\end{equation}

\noindent Thus we obtain the following system of differential equations for
unknown coefficients $a_{k}(p,r)$, \,$b_{k}(p,r)$ ($k=1,2,...$)
\begin{equation}\label{20}\begin{cases}
r (b_1(p,r))^\prime_r + b_1(p,r)=-2r\,(Mf(p,r))^\prime_p\\
r (b_2(p,r))^\prime_r + 2\,b_2(p,r)+2\,r (a_{1}(p,r))^\prime_p =0\\
r (a_2(p,r))^\prime_r + 2\,a_2(p,r)-2\,r (b_{1}(p,r))^\prime_p =2r\,(Mf(p,r))^\prime_r,\end{cases}
\end{equation}
for  $k=1$, $k=2$ and
\begin{equation}\label{21}\begin{cases}
r (b_k(p,r))^\prime_r + k\,b_k(p,r)-r (b_{k-2}(p,r))^\prime_r+(k-2)b_{k-2}(p,r)+2\,r (a_{k-1}(p,r))^\prime_p =0
\\
r (a_k(p,r))^\prime_r + k\,a_k(p,r)-r (a_{k-2}(p,r))^\prime_r+(k-2)a_{k-2}(p,r)-2\,r (b_{k-2}(p,r))^\prime_p =0,
\end{cases}
\end{equation}
for $k>2$.

\section{The consistent solution}\label{S4}

\noindent  Returning to the formula \eqref{12} for every $(x,y)$
we have
\begin{equation}\label{22}
f(x,y)=f(x,y,\phi)=Mf(p,r)+\sum_{k=1}^{\infty} \left(a_{k}(p,r) \cos{k\,\varphi}  + b_{k}(p,r) \sin{k\,\varphi} \right).\end{equation}
 $\phi=\pi/2$ corresponds to $\varphi=0$ the case that $p\in L$ is the projection of $(x,y)$ onto $L$, hence we have $p=x$. Thus for $p=x$ and $r=y$ from \eqref{22} we have
\begin{equation}\label{23}
f(x,y)=Mf(p,r)+\sum_{k=1}^{\infty} a_k(p,r).\end{equation}

\noindent  Now the problem is to calculate
\begin{equation}\label{24}
\sum_{k=1}^{\infty} a_k(p,r).\end{equation}

\noindent Taking into account that $f$ supported in the compact region $G$ located on one side of the line $L$ for any $(p,r)$ we have
$$
0=f(p,r,-\pi)=f(p,r,\pi)=Mf(p,r)+\sum_{k=1}^{\infty} \left(a_{k}(p,r) \cos{k\pi}  + b_{k}(p,r) \sin{k\pi}\right),$$ hence

\begin{equation}\label{24.1}Mf(p,r)+\sum_{k=1}^{\infty} (-1)^k\,a_{k}(p,r)=0.\end{equation}

\noindent From \eqref{23} and \eqref{24.1} we have
\begin{equation}\label{24.2}
f(x,y)=2\,Mf(p,r)+\sum_{k=1}^{\infty} a_{2k}(p,r),\end{equation}
where $p=x$ and $r=y$.

\noindent  Now the problem is to calculate
\begin{equation}\label{24.3}
\sum_{k=1}^{\infty} a_{2k}(p,r)\end{equation}

\noindent Taking the sums of the second equations of \eqref{21} for even $k$ we get the following recurrent equations for $a_{2k}(p,r)$.
\begin{equation}\label{25}\begin{cases}
r (a_2(p,r))^\prime_r + 2\,a_2(p,r)-2\,r (b_{1}(p,r))^\prime_p =2r\,(Mf(p,r))^\prime_r,\\
r (a_{2k}(p,r))^\prime_r + 2k\,a_{2k}(p,r) + \sum_{j=1}^{k-1}4j\,a_{2j}(p,r)-2\,r \sum_{j=1}^{k}(b_{2j-1}(p,r))^\prime_p \\\quad\quad =2r\,(Mf(p,r))^\prime_r\,\,\,\,\,\texttt{for} \,\,\,\,k>1\end{cases}
\end{equation}
and we see that to calculate $a_{2k}(p,r)$ we need to known  coefficients $b_{k}(p,r)$ for odd $k$.

\noindent Taking the sums of the first equations of  \eqref{21} for odd $k$ we get the following recurrent equations for $b_{2k-1}(p,r)$.
\begin{equation}\label{26}\begin{cases}
r (b_1(p,r))^\prime_r + b_1(p,r)=-2r\,(Mf(p,r))^\prime_p,\\
r (b_{2k-1}(p,r))^\prime_r + (2k-1)\,b_{2k-1}(p,r) + \sum_{j=1}^{k-1}2(2j-1)\,b_{2j-1}(p,r)\\\quad\quad +2\,r \sum_{j=1}^{k-1}(a_{2j}(p,r))^\prime_p =-2r\,(Mf(p,r))^\prime_p\,\,\,\,\, \texttt{for} \,\,\,\,k>1.\end{cases}
\end{equation}
Using equations \eqref{25} and \eqref{26} one can calculate step by step the unknown coefficients $a_{2k}(p,r)$  and the unknown coefficients $b_{2k-1}(p,r)$  (at first we find $b_1$, next  $a_2$, next $b_3$ next $a_4$ and so on).

\section{Solution of the differential equations}\label{S5}

\noindent Now we are going to find a
boundary conditions for the differential equations. We have for $k\geq1$
\begin{equation}\label{34}
a_{2k}(p,r)=\frac{1}{\pi}\int_{-\pi}^{\pi}f(p,r,\varphi)\cos{2k\varphi}\,d\,\varphi\end{equation} and
\begin{equation}\label{35}
b_{2k-1}(p,r)=\frac{1}{\pi}\int_{-\pi}^{\pi}f(p,r,\varphi)\sin{(2k-1)\varphi}\,d\,\varphi.\end{equation}
Taking into account that $f$ supported in the compact region $G$ located on one side of the line $L$ we get the following boundary conditions
\begin{equation}\label{36}
a_{2k}(p,0)=0  \,\,\,\emph{and}\,\,\,\,b_{2k-1}(p,0)=0 \,\,\,\,\, \emph{for}\,\,\,\,\,p\in L.\end{equation}

\noindent Thus we get two systems of differential equations \eqref{25} and \eqref{26} with boundary conditions \eqref{36}.
The unique solution of \eqref{25} for $k\geq1$ is
$$
a_{2k}(p,r)=
$$
\begin{equation}\label{36.1}\frac{1}{r^{2k}}\int_{0}^{r}u^{2k-1} \left(2u (Mf(p,u))'_u- 2\sum_{j=1}^{k-1}2j a_{2j}(p,u)+2u\sum_{j=1}^{k}(b_{2j-1}(p,r))^\prime_p\right)d\,u.\end{equation}

The unique solution for $b_1(p,r)$ is
\begin{equation}\label{37}
b_{1}(p,r)=
\frac{1}{r}\int_{0}^{r}-2\,u (Mf(p,u))'_p\,d\,u.\end{equation}

The unique solution of \eqref{26} for $k\geq2$ is
$$
b_{2k-1}(p,r)=
$$
\begin{equation}\label{38}\frac{1}{r^{2k-1}}\int_{0}^{r}u^{2k-2} \left(-2\,u (Mf(p,u))'_p- 2\sum_{j=1}^{k-1}(2j-1) b_{2j-1}(p,u)-2u\sum_{j=1}^{k-1}(a_{2j}(p,r))^\prime_p\right)d\,u.\end{equation}

\noindent Note that it follows from  \eqref{24.2}, \eqref{36.1} and \eqref{38} that for $(x,y)\in G$ the value $f(x,y)$ depends on values Mf on a neighborhood of $p=(x,0)\in L$ and $0\leq r\leq y$. Theorem 1 is proved.

\begin{lemma}\label{1} There are sequences of polynomials defined on $[0,1]$

1) $A_{2k,2i}$ of degree $2k-1+2i$ for integers $k\geq1$ and $0\leq i\leq k$
\begin{equation}\label{39}A_{2k,2i}(t)=\sum_{j=1}^{k+i}A_j(2k,2i)\,t^{2j-1},\end{equation}

2) $B_{2k-1,2i-1}$ of degree $2k+2i-3$ for integers $k\geq1$ and $1\leq i\leq k$
\begin{equation}\label{40}B_{2k-1,2i-1}(t)=\sum_{j=1}^{k+i-1}B_j(2k-1,2i-1)\,t^{2j-1},\end{equation}
such that
\begin{equation}\label{41}
a_{2k}(p,r)=2 Mf(p,r)+ \int_{0}^{r}\sum_{i=0}^{k}r^{2i-1}\,A_{2k,2i}(u/r)\,(Mf(p,u))^{(2i)}\, du
\end{equation}
and
\begin{equation}\label{42}
b_{2k-1}(p,r)=\int_{0}^{r}\sum_{i=1}^{k}r^{2(i-1)}\,B_{2k-1,2i-1}(u/r)\,(Mf(p,u))^{(2i-1)}\, du
\end{equation}
here and below $(Mf(p,u))^{(j)}$  is the derivative of order $j$ with respect the variable $p$  ($(Mf(p,u))^{(0)}=Mf(p,u)$).
\end{lemma}
\begin{proof} Mathematical induction  can be used to prove Lemma 1. It follows from \eqref{37} that for $k=1$ \eqref{41} and \eqref{42} are true. Indeed \begin{equation}\label{43}
b_{1}(p,r)=\int_{0}^{r} (-2\frac{u}{r})(Mf(p,u))^{(1)}_p\, du
\end{equation}
from \eqref{36.1} using \eqref{43} we have
$$
a_{2}(p,r)=\frac{1}{r^{2}}\int_{0}^{r}u\left(2u (Mf(p,u))'_u+2u (b_{1}(p,u))^\prime_p\right)d\,u=$$
\begin{equation}\label{44}2 Mf(p,r)+ \int_{0}^{r}\left(r^{-1}\,(-4\frac{u}{r})(Mf(p,u))+r(-2\frac{u}{r}(1-(\frac{u}{r})^2))(Mf(p,u))^{(2)}\right)\, du.\end{equation}
Suppose \eqref{41} and \eqref{42} are true for some $n = k$. Prove that  \eqref{41} and \eqref{42} are true for $n = k + 1$.
From \eqref{38} we have
$$
b_{2k+1}(p,r)=
$$
\begin{equation}\label{45}\frac{1}{r^{2k+1}}\int_{0}^{r}u^{2k} \left(-2\,u (Mf(p,u))'_p- 2\sum_{j=1}^{k}(2j-1) b_{2j-1}(p,u)-2u\sum_{j=1}^{k}(a_{2j}(p,u))^\prime_p\right)d\,u.\end{equation}
Substituting the expressions for $b_{2j-1}(p,u)$ and $a_{2j}(p,u)$  from \eqref{41} and \eqref{42}  into \eqref{45} we obtain
\begin{multline}\label{46}
b_{2k+1}(p,r)=-2\int_{0}^{r}(\frac{u}{r})^{2k+1}(Mf(p,u))'_p\,du -
2\int_{0}^{r}\sum_{j=1}^{k}\frac{u^{2k}(2j-1)}{r^{2k+1}}
\int_{0}^{u}\sum_{i=1}^{j}u^{2(i-1)}\\
\times\left( \sum_{m=1}^{j+i-1}B_m(2j-1,2i-1)\,(\frac{v}{u})^{2m-1} \right)(Mf(p,v))^{(2i-1)}\,dv\,
du
-2\int_{0}^{r}\sum_{j=1}^{k}(\frac{u}{r})^{2k+1}\\\times\left(2 Mf(p,u))'_p+\int_{0}^{u}\sum_{i=1}^{j}u^{2i-1}\left(\sum_{m=1}^{j+i}A_m(2j,2i)(\frac{v}{u})^{2m-1} \right)(Mf(p,v))^{(2i-1)}dv\right)du.
\end{multline}
Changing the order of summation in \eqref{46} and the order of integration we get
\begin{multline}\label{47}
b_{2k+1}(p,r)=-2\int_{0}^{r}(\frac{u}{r})^{2k+1}(Mf(p,u))'_p\,du -
\int_{0}^{r}\sum_{i=1}^{k}r^{2(i-1)}
\sum_{j=i}^{k}(2j-1)\\
\times \sum_{m=1}^{j+i-1}\frac{B_m(2j-1,2i-1)}{k+i-m}\left((\frac{u}{r})^{2m-1} -(\frac{u}{r})^{2k+2i-1}\right)(Mf(p,v))^{(2i-1)}\,du
\\-4k\int_{0}^{r}(\frac{u}{r})^{2k+1}(Mf(p,u))'_p\,du  -\int_{0}^{r}\sum_{i=1}^{k+1}r^{2(i-1)}
\sum_{j=i-1}^{k}\\
\times \sum_{m=1}^{j+i-1}\frac{A_m(2j,2(i-1))}{k+i-m}\left((\frac{u}{r})^{2m-1} -(\frac{u}{r})^{2k+2i-1}\right)(Mf(p,v))^{(2i-1)}\,du.
\end{multline}
After grouping of summands finally we obtain
\begin{equation}\label{48}
b_{2k+1}(p,r)=\int_{0}^{r}\sum_{i=1}^{k+1}r^{2(i-1)}\,B_{2k+1,2i-1}(u/r)\,(Mf(p,u))^{(2i-1)}\, du,
\end{equation}
where for $i=1$
\begin{multline}\label{48.1}
B_{2k+1,1}=-\sum_{m=1}^{k}\left( \sum_{j=m}^{k}\frac{(2j-1)B_m(2j-1,1)+A_m(2j,0)}{k+1-m})\right)(\frac{u}{r})^{2m-1}+\\
\left(\sum_{j=1}^{k}\sum_{m=1}^{j}\frac{(2j-1)B_m(2j-1,1)+A_m(2j,0)}{k+1-m}-2(2k+1)\right)(\frac{u}{r})^{2k+1}
\end{multline}
with
\begin{equation}\label{48.2}\begin{cases}
B_m(2k+1,1)=-\sum_{j=m}^{k}\frac{(2j-1)B_m(2j-1,1)+A_m(2j,0)}{k+1-m} \,\, \texttt{for}\,\,\, 1\leq m\leq k\\
B_{k+1}(2k+1,1)=\sum_{j=1}^{k}\sum_{m=1}^{j}\frac{(2j-1)B_m(2j-1,1)+A_m(2j,0)}{k+1-m}-2(2k+1);
\end{cases}
\end{equation}
for $1<i\leq k$
\begin{multline}\label{48.3}
B_{2k+1,2i-1}=-\sum_{j=i}^{k}\sum_{m=1}^{j+i-1}\frac{(2j-1)B_m(2j-1,2i-1)+A_m(2j,2(i-1))}{k+i-m}(\frac{u}{r})^{2m-1}-\\
\sum_{m=1}^{2i-2}\frac{A_m(2(i-1),2(i-1))}{k+i-m}(\frac{u}{r})^{2m-1}+\\
\left(\sum_{j=i}^{k}\sum_{m=1}^{j+i-1}\frac{(2j-1)B_m(2j-1,2i-1)+A_m(2j,2(i-1))}{k+i-m}+\right.\\\left.\sum_{m=1}^{2i-2}\frac{A_m(2(i-1),2(i-1))}{k+i-m}\right)(\frac{u}{r})^{2k+2i-1}
\end{multline}
with
\begin{equation}\label{48.4}\begin{cases}
B_m(2k+1,2i-1)=-\sum_{j=i}^{k}\frac{(2j-1)B_m(2j-1,2i-1)+A_m(2j,2(i-1))}{k+i-m}-\\
\quad\quad\frac{A_m(2(i-1),2(i-1))}{k+i-m}\,\, \texttt{for}\,\,\, 1\leq m\leq 2i-2\\
B_m(2k+1,2i-1)=-\sum_{j=m-i+1}^{k}\frac{(2j-1)B_m(2j-1,2i-1)+A_m(2j,2(i-1))}{k+i-m} \\
\quad\quad \texttt{for}\,\,\, 2i-1\leq m\leq k+i-1\\
B_{k+i}(2k+1,2i-1)=\sum_{j=i}^{k}\sum_{m=1}^{j+i-1}\frac{(2j-1)B_m(2j-1,2i-1)+A_m(2j,2(i-1))}{k+i-m}+\\
\quad\quad+ \sum_{m=1}^{2i-2}\frac{A_m(2(i-1),2(i-1))}{k+i-m};
\end{cases}
\end{equation}
for $i=k+1$
\begin{equation}\label{48.5}
B(2k+1,2k+1)=-\sum_{m=1}^{2k}\frac{A_m(2k,2k)}{2k+1-m}\left((\frac{u}{r})^{2m-1}-(\frac{u}{r})^{4k+1}\right)
\end{equation}
with
\begin{equation}\label{48.6}\begin{cases}
B_m(2k+1,2k+1)=-\frac{A_m(2k,2k)}{2k+1-m}\,\, \texttt{for}\,\,\, 1\leq m\leq 2k\\
B_{2k+1}(2k+1,2k+1)=\sum_{m=1}^{2k}\frac{A_m(2k,2k)}{2k+1-m}.
\end{cases}
\end{equation}

\noindent Now lets prove Lemma 1 for $a_{2(k+1)}$. From \eqref{36.1} we have
\begin{multline}\label{48.7}
a_{2(k+1)}(p,r)=\\
\frac{1}{r^{2(k+1)}}\int_{0}^{r}u^{2k+1} \left(2u (Mf(p,u))'_u- 2\sum_{j=1}^{k}2j a_{2j}(p,u)+2u\sum_{j=1}^{k+1}(b_{2j-1}(p,r))^\prime_p\right)d\,u.\end{multline}
Substituting the expressions for $b_{2j-1}(p,u)$ and $a_{2j}(p,u)$  from \eqref{41} and \eqref{42}  into \eqref{48.7} we obtain
\begin{multline}\label{48.8}
a_{2(k+1)}(p,r)=\frac{2}{r^{2(k+1)}}\int_{0}^{r}u^{2(k+1)}(Mf(p,u))'_u\,du -
\frac{2}{r^{2(k+1)}}\int_{0}^{r}\sum_{j=1}^{k}2j\,u^{2(k+1)}\times\\
\left(2 Mf(p,u)+ \int_{0}^{u}\sum_{i=0}^{j}u^{2i-1}\left(\sum_{m=1}^{j+i}A_m(2j,2i)(\frac{v}{u})^{2m-1} \right)(Mf(p,v))^{(2i)}dv\right)du+
\\2\int_{0}^{r}\sum_{j=1}^{k+1}(\frac{u}{r})^{2(k+1)} \int_{0}^{u}\sum_{i=1}^{j}u^{2i-1}\left(\sum_{m=1}^{j+i-1}B_m(2j-1,2i-1)(\frac{v}{u})^{2m-1} \right)(Mf(p,v))^{(2i)}dv\,du.
\end{multline}
Changing the order of summation in \eqref{48.8} and the order of integration we get
\begin{multline}\label{48.9}
a_{2(k+1)}(p,r)=2Mf(p,r)-\int_{0}^{r}4(k+1)^2r^{-1}(\frac{u}{r})^{2k+1}Mf(p,u)\,du -
\\\int_{0}^{r}r^{-1}\sum_{j=1}^{k}\sum_{m=1}^{j}\frac{2j\,A_m(2j,0)}{k-m+1}\left((\frac{u}{r})^{2m-1} -(\frac{u}{r})^{2k+1}\right)Mf(p,u)\,du-
\\\int_{0}^{r}\sum_{i=1}^{k}r^{2i-1}\sum_{j=i}^{k}\sum_{m=1}^{j+i}\frac{2j\,A_m(2j,2i)}{k+i-m+1}\left((\frac{u}{r})^{2m-1} -(\frac{u}{r})^{2k+2i+1}\right)(Mf(p,u))^{(2i)}\,du+
\\\int_{0}^{r}\sum_{i=1}^{k}r^{2i-1}\sum_{j=i}^{k}\sum_{m=1}^{j+i-1}\frac{B_m(2j-1,2i-1)}{k+i-m+1}\left((\frac{u}{r})^{2m-1} -(\frac{u}{r})^{2k+2i+1}\right)(Mf(p,u))^{(2i)}\,du+
\\\int_{0}^{r}\sum_{i=1}^{k+1}r^{2i-1}\sum_{m=1}^{k+i}
\frac{B_m(2k+1,2k+1)}{k+i-m+1}\left((\frac{u}{r})^{2m-1} -(\frac{u}{r})^{2k+2i+1}\right)(Mf(p,u))^{2i}\,du.
\end{multline}
After grouping of summands finally we obtain
\begin{equation}\label{49}
a_{2(k+1)}(p,r)=2 Mf(p,r)+ \int_{0}^{r}\sum_{i=0}^{k+1}r^{2i-1}\,A_{2(k+1),2i}(u/r)\,(Mf(p,u))^{(2i)}\, du,
\end{equation}
where for $i=0$
\begin{multline}\label{49.1}
A_{2(k+1),0}=-4(k+1)^2(\frac{u}{r})^{2k+1}-\sum_{j=1}^{k}
\sum_{m=1}^{j}\frac{2j\,A_m(2j,0)}{k-m+1}\left((\frac{u}{r})^{2m-1} -(\frac{u}{r})^{2k+1}\right)
\end{multline}
with
\begin{equation}\label{50}\begin{cases}
A_m(2(k+1),0)=-\sum_{j=m}^{k}
\frac{2j\,A_m(2j,0)}{k-m+1} \,\, \texttt{for}\,\,\, 1\leq m\leq k\\
A_{k+1}(2(k+1),0)=-4(k+1)^2+\sum_{m=1}^{k}\sum_{j=m}^{k}\frac{2j\,A_m(2j,0)}{k-m+1};
\end{cases}
\end{equation}
for $1\leq i\leq k$
\begin{multline}\label{51}
A_{2(k+1),2i}=
\sum_{j=i}^{k+1}\sum_{m=1}^{j+i-1}\frac{B_m(2j-1,2i-1)}{k+i-m+1}\left((\frac{u}{r})^{2m-1} -(\frac{u}{r})^{2k+2i+1}\right)-\\
\sum_{j=i}^{k}\sum_{m=1}^{j+i}\frac{2j\,A_m(2j,2i)}{k+i-m+1}\left((\frac{u}{r})^{2m-1} -(\frac{u}{r})^{2k+2i+1}\right)
\end{multline}
with
\begin{equation}\label{52}\begin{cases}
A_m(2(k+1),2i)=\sum_{j=i}^{k}\left(\frac{B_m(2j+1,2i-1)-2j\,A_m(2j,2i)+B_m(2i-1,2i-1)}{k+i-m+1}\right)\,\, \texttt{for}\,\,\,1\leq m\leq 2i-1\\
A_m(2(k+1),2i)=\sum_{j=m-i}^{k}\left(\frac{B_m(2j+1,2i-1)-2j\,A_m(2j,2i)}{k+i-m+1}\right) \texttt{for}\,\,\, 2i\leq m\leq k+i\\
A_{k+i+1}(2(k+1),2i)=\sum_{j=i}^{k}\sum_{m=1}^{j+i}\frac{2j\,A_m(2j,2i)-B_m(2j+1,2i-1)}{k+i-m+1}-\\
\quad\quad\quad \sum_{m=1}^{2i-1}\frac{B_m(2i-1,2i-1)}{k+i-m+1};
\end{cases}
\end{equation}
for $i=k+1$
\begin{equation}\label{53}
A(2(k+1),2(k+1))=\sum_{m=1}^{2k+1}\frac{B_m(2k+1,2k+1)}{2k-m+2}\left((\frac{u}{r})^{2m-1} -(\frac{u}{r})^{4k+3}\right)
\end{equation}
with
\begin{equation}\label{54}\begin{cases}
A_m(2(k+1),2(k+1))=\frac{B_m(2k+1,2k+1)}{2k-m+2}\,\, \texttt{for}\,\,\, 1\leq m\leq 2k+1\\
A_{2(k+1)}(2(k+1),2(k+1))=-\sum_{m=1}^{2k+1}\frac{B_m(2k+1,2k+1)}{2k-m+2}
\end{cases}
\end{equation}
\end{proof}
Also note that we obtain recurrent relations \eqref{48.2}, \eqref{48.4}, \eqref{48.6}, \eqref{50}, \eqref{52}, \eqref{54} between coefficients  $A_j(2k,2i)$  and $B_j(2k-1,2i-1)$ for $k>1$  and $i\leq k$.

\noindent The first few  coefficients are:
\begin{equation}\label{55}\begin{cases}
B_1(1,1)=-2\\
A_1(2,0)=-4, \,\,\, A_1(2,2)=-2, \,\,\,A_2(2,2)=2 \\
B_1(3,1)=6, \,\,\, B_2(3,1)=-12\\
B_1(3,3)=1,\,\,\, B_2(3,3)=-2,\,\,\, B_3(3,3)=1\\
A_1(4,0)=8, \,\,\, A_2(4,0)=-24\\
A_1(4,2)=4, \,\,\, A_2(4,2)=-16, \,\,\,A_3(4,2)=12 \\
A_1(4,4)=1/3, \,\,\, A_2(4,4)=-1, \,\,\,A_3(4,4)=1,\,\,\,A_4(4,4)=-1/3 \\
B_1(5,1)=-10,\,\,\, B_2(5,1)=60,\,\,\, B_3(5,1)=-60\\
B_1(5,3)=-2/3,\,\,\, B_2(5,3)=10,\,\,\, B_3(5,3)=-15,\,\,\,B_4(5,3)=20/3 \\
B_1(5,5)=-1/{12},\,\,\, B_2(5,5)=1/3,\,\,\, B_3(5,5)=-1/2,\,\,\,B_4(5,5)=1/3,\,\,\,B_5(5,5)=-1/12 \\
......
\end{cases}
\end{equation}
Using recurrent relations one can calculate all coefficients of  $A(2n,2i)$  and

\noindent $B(2n-1,2i-1)$ for $i\leq n$ by means of coefficients $A(2k,2i)$  and $B(2k-1,2i-1)$ for $1\leq k<n$.

\section{Partial sums of the series}\label{S6}

\noindent We are going to consider the partial sums of the series (see \eqref{24.3})
\begin{equation}\label{60}
\sum_{k=1}^{\infty} a_{2k}(p,r)\end{equation}
Taking into account  \eqref{41} we have
\begin{multline}\label{61}
\sum_{k=1}^{n} a_{2k}(p,r)=\sum_{k=1}^{n}\left(2 Mf(p,r)+ \int_{0}^{r}\sum_{i=0}^{k}r^{2i-1}\,A_{2k,2i}(u/r)\,(Mf(p,u))^{(2i)}\, du\right)=\\
2n Mf(p,r)+\sum_{k=1}^{n} \int_{0}^{r}\sum_{i=0}^{k}r^{2i-1}\left(A_{2k,2i}(u/r)\right)(Mf(p,u))^{(2i)}\, du.\end{multline}
Changing the order of summation in \eqref{61} we obtain
\begin{equation}\label{62}
\sum_{k=1}^{n} a_{2k}(p,r)=
2n Mf(p,r)+\sum_{i=0}^{n}\int_{0}^{r}r^{2i-1}\left(\sum_{k=i}^{n}A_{2k,2i}(u/r)\right)(Mf(p,u))^{(2i)}du.\end{equation}
Note that here we assume $A_{0,0}\equiv0$

\noindent We denote by $Z_{n,i},$ $0\leq i\leq n$ the following polynomial of degree $2n+2i-1$ defined on the interval $[0,1]$
\begin{equation}\label{63}
Z_{n,i}(t)= \sum_{k=i}^{n}A_{2k,2i}(t)=\sum_{j=1}^{n+i}z_j(2k,2i)\,t^{2j-1}.\end{equation}
Substituting \eqref{39} into \eqref{63} we obtain
\begin{equation}\label{64}
Z_{n,i}(t)= \sum_{k=i}^{n}A_{2k,2i}(t)=\sum_{k=i}^{n}\sum_{j=1}^{k+i}A_j(2k,2i)\,t^{2j-1}=\sum_{j=1}^{n+i}\left(\sum_{k=i, j-i\leq k}^{n}A_j(2k,2i)\right)t^{2j-1}.\end{equation}
Substituting \eqref{62} and \eqref{64} into  \eqref{24.2} we obtain
\begin{multline}\label{65}
f(x,y)=2\,Mf(p,r)+\lim_{n\to \infty}\sum_{k=1}^{n}  a_{2k}(p,r)=\\\lim_{n\to \infty}\left(2(n+1)Mf(p,r)+\sum_{i=0}^{n}\int_{0}^{r}r^{2i-1}Z_{n,i}(u/r)(Mf(p,u))^{(2i)}du\right)\end{multline}
where $p=x$, $\,r=y$ and
\begin{equation}\label{66}Z_{n,i}(t)=\sum_{j=1}^{n+i}z_j(2n,2i)\,t^{2j-1}\end{equation}
are polynomials with coefficients
\begin{equation}\label{67}
z_j(2n,2i)=\sum_{k=i, j-i\leq k}^{n}A_j(2k,2i)\,\,\,\,\,\emph{for}\,\,\,\,\,0\leq i\leq n. \end{equation}
Note that one can find the coefficients $A_j(2k,2i)$ from recurrent relations \eqref{48.2}, \eqref{48.4}, \eqref{48.6}, \eqref{50}, \eqref{52}, \eqref{54}.
Theorem 2 is proved.

\section{Implementation of the reconstruction
formula}\label{S6}

The problem of reconstructing a function from spherical means is important for many
imaging and remote sensing applications (see, for example, \cite{FiRa}, \cite{KuKu}, \cite{AmPa}). These
applications require inversion of the  spherical
Radon transform. However, those formulas require continuous data, whereas in practical
applications only a discrete data set is available. In some tomographic applications iterative reconstruction
algorithms are more common. In spite of absence of exact FBP formulas in 2D, approximate ones
that preserve all the singularities of the image can be easily written
and then improved by successive iterative corrections.
However, due to the presence of the derivative, the inversion formulas are sensitive to error in
the data $Mf$ (see \cite{Nat}).

\noindent In the present paper, we have established a new
iterative reconstruction
algorithm to recover a function $f$ supported in a compact region from
its spherical means $Mf$ (see Theorem 2) which is different from the existing ones in  \cite{FiHa}, \cite{FiPa},  \cite{Hal1}.
Our reconstruction
formula can be numerically implemented due to a local description.
Thus, when evaluating $f(x,y)$ for $(x,y)\in B(O,R)$ we just compute the integral
for the set of frequencies uniformly distributed over the interval $[0,y]$. However, due to the presence of the derivatives of higher order, the inversion formulas are sensitive to error in the data $Mf$. In some cases for the derivatives one can use their analytic expressions.
To estimate the iteration speed we use the following known result from the  theory of Fourier series expansion.
\emph{Let f be $2\pi$-periodic, continuous, with piecewise-continuous first-derivative function.
Then the Fourier series of $f $ converges uniformly
\begin{equation}\label{56}
\sup_{x\in\mathbf R^1}|f(x)-S_n(f,x)|\leq c\frac{\ln n}{n},
\end{equation}
where $S_n(f,x)$  is the partial sum of  the Fourier series of $f $ and  $c$  does not depend on $n$.}

\noindent Note that, we get not only uniform convergence, but also a rate of convergence.
Thus by finding polynomials $Z_{n,i}$ (see \eqref{9}) for large $n$ one can recover a function $f$ by approximation as close as we want.

\end{document}